\documentclass[12pt]{amsart}
\usepackage{amsmath}
\usepackage{amsfonts}
\usepackage{amssymb}
\usepackage[utf8]{inputenc}
\usepackage{amssymb}

\usepackage{enumerate}

\usepackage{graphicx}

\newtheorem{thrm}{Theorem}[section]
\newtheorem{cor}[thrm]{Corollary}
\newtheorem{lemma}[thrm]{Lemma}

\theoremstyle{definition}
\newtheorem{defin}[thrm]{Definition}

\newtheorem*{xrem}{Remark}

\numberwithin{equation}{section}

\DeclareMathOperator{\supp}{supp}

\newcommand{\NN}{\mathbb{N}}
\newcommand{\RR}{\mathbb{R}}

\begin{document}

%%%%% To ease editing, for IMPAN journals add:

\baselineskip=17pt

%%%%%%%%%%%

%% In the running head, replace first names by initials 
%% and give an abbreviation of the title.

\title[On functional tightness]{On functional tightness of infinite products}

\author[M. Krupski]{Miko\l aj Krupski}
\address{Department of Mathematics\\ University of Pittsburgh\\
Pittsburgh, PA 15260, USA\\
and \\
Institute of Mathematics\\ University of Warsaw\\ ul. Banacha 2\\
02--097 Warszawa, Poland }
\email{m.krupski@pitt.edu}

\begin{abstract}
A classical theorem of Malykhin says that if $\{X_\alpha:\alpha\leq\kappa\}$ is a
family of compact spaces such that $t(X_\alpha)\leq \kappa$, for every $\alpha\leq\kappa$, then
$t\left( \prod_{\alpha\leq \kappa} X_\alpha \right)\leq \kappa$, where $t(X)$ is the tightness of a space $X$.
In this paper we
prove the following counterpart of Malykhin's theorem for functional tightness: 
Let $\{X_\alpha:\alpha<\lambda\}$ be a family of compact spaces
such that $t_0(X_\alpha)\leq \kappa$ for every $\alpha<\lambda$. If $\lambda \leq 2^\kappa$ or $\lambda$ is less than the first measurable cardinal,
then $t_0\left( \prod_{\alpha<\lambda} X_\alpha \right)\leq \kappa$, where $t_0(X)$ is the functional tightness
of a space $X$. In particular, if there are no measurable cardinals, then the functional tightness is preserved by arbitrarily large products of compacta. Our
result answers a question posed by Okunev.
\end{abstract}



\maketitle

\section{Introduction}
One of the important cardinal invariants of a topological space $X$ is its tightness $t(X)$\
\footnote{\textit{The tightness} $t(X)$ of a space $X$ is the minimal infinite cardinal number $\kappa$ such that
for any $A\subseteq X$ and any $x\in \overline{A}$ there is $C\subseteq A$ with $|C|\leq \kappa$ and $x\in \overline C$}.
In this paper we will be interested in the following well-known
modification of the tightness arising naturally from the theory of function spaces: Recall that
a function $f:X\to Z$ is
\textit{$\kappa$-continuous} if its restriction to any subset of $X$ of cardinality $\leq\kappa$ is continuous.
By $t_0(X)$ we denote the \textit{functional tightness} of $X$, i.e. the minimal infinite cardinal number $\kappa$ such that any
$\kappa$-continuous
real-valued function on $X$
is continuous.

A classical theorem of V.I.~Malykhin \cite{M}
(cf. \cite[2.3.3]{A1}, \cite[5.9]{J}) asserts that the tightness behaves nicely under infinite Cartesian products of compact spaces.
The purpose of the present note is to prove the following counterpart of Malykhin's theorem for functional tightness:
\begin{thrm}\label{main}
Let $\kappa$ be an infinite cardinal and let $\{X_\alpha:\alpha<2^\kappa\}$ be a family of compact spaces such that $t_0(X_\alpha)\leq\kappa$, for every
$\alpha<2^\kappa$.
Then $t_0\left(\prod_{\alpha<2^\kappa} X_\alpha\right)\leq\kappa$.
\end{thrm}
This answers a question posed recently by O.\ Okunev (see \cite[Question 3.2]{O}).
In fact, we establish a more general result (cf. Corollary \ref{main1} below), which in particular asserts that if there are no measurable cardinals, then the
functional tightness is preserved by arbitrarily large products of compact spaces.

The functional tightness $t_0(X)$ of a space $X$ is an 
interesting cardinal invariant which is related to both the tightness $t(X)$
and the density character $d(X)$ (i.e. the minimal cardinality of a dense subspace in $X$) of $X$.
While the tightness $t(X)$ measures the minimal cardinality of sets required to determine the topology of $X$, the functional tightness
measures the minimal size of sets required to guarantee the continuity of real-valued functions on $X$.
It is easy to see that always $t_0(X)\leq t(X)$, but in fact the following holds:
$t(X)=\sup\{t_0(Y):Y\subseteq X\}$, for every space $X$ (see \cite[Theorem 1]{A}). On the other hand, the functional tightness is not monotonic
with respect to
closed subspaces which is in a sharp contrast with the behavior of the tightness \cite[Example 2]{A}.

The main motivation for studying functional tightness comes from the theory of function spaces where it appears naturally in
the context of realcompactness (see \cite{A} \cite{U}, \cite{Corson}, \cite{Pl1}, \cite{Pl}).
Though the notion was introduced by A.V.\ Arhangel'skii \cite{A} in his study of the Hewitt number of $C_p(X)$-spaces, the concept of
countable functional tightness
appeared much earlier in a paper of H.H.\ Corson \cite{Corson}
(in the context of realcompactness of Banach spaces in the weak topology).

The question whether there is a counterpart of Malykhin's theorem for functional tightness was asked recently by O.\ Okunev \cite[Question 3.2]{O},
who studied the behavior of (weak) functional tightness under finite products (see \cite{O}). Our approach relies on his result and a theorem of N.\ Noble
concerning the continuity of functions on Cartesian products
(see Corollary \ref{Okunev} and Theorem \ref{Noble} below).

\section{Notation and auxiliary facts}

All spaces under consideration are Hausdorff.

\subsection{Functional tightness and minitightness}
Let us recall definitions of functional tightness and a related notion of minitightness, and their basic properties. To this end, it will
be convenient to introduce the following useful terminology:
\begin{defin}\label{def.kappa-cont}
Let $X$ and $Z$ be topological spaces and let $\kappa\geq \omega$ be a cardinal number. A function $f:X\to Z$ is called:
\begin{itemize}
\item \textit{$\kappa$-continuous} if its restriction to any subset of $X$ of cardinality $\leq\kappa$ is continuous.
\item \textit{strictly $\kappa$-continuous} if for any subset $A\subseteq X$ of cardinality $\leq\kappa$, there is a continuous function $\hat{f}:X\to Z$ such that
 $\hat{f}\upharpoonright A=f\upharpoonright A$.
\end{itemize}
\end{defin}
Of course, strict $\kappa$-continuity implies $\kappa$-continuity. It is also not difficult to show that both notions coincide if $X$ is a normal space and
$Z=\RR$ (cf. \cite[Theorem 3]{A}).

\medskip

The \textit{functional tightness} of a space $X$ is the cardinal number
 $$t_0(X)=\min\{\kappa\geq \omega:\text{ every $\kappa$-continuous real-valued function on $X$ is continuous}\}.$$
Similarly, the \textit{minitightness}\footnote{Also called weak functional tightness, $\mathbb{R}$-tightness \cite[p. 59]{Arh} or modified functional tightness
\cite{U}}
of a space $X$ is the cardinal number
 $$t_m(X)=\min\{\kappa\geq \omega:\text{ every strictly $\kappa$-continuous real-valued function on $X$ is continuous}\}.$$
Obviously, for any space $X$ we have $t_m(X)\leq t_0(X)$. It is also easy to prove that for any space $X$, we have the inequality
$t_0(X)\leq\min\{t(X),d(X)\}$, where $t(X)$ is
the tightness and $d(X)$ is the density character of $X$ (cf. \cite[Corollary 1]{A}).
Since, as we have already mentioned, the notions of $\kappa$-continuity and strict $\kappa$-continuity of real-valued functions coincide
for normal spaces, we have
$t_m(X)=t_0(X)$ provided $X$ is a normal space (cf. \cite[Theorem 3]{A}).

\subsection{Products and $\sigma$-products}

Let $\{X_\alpha:\alpha\in A\}$ be a family of topological spaces and let $a=(a_\alpha)_{\alpha\in A}$ be a point in the product $X=\prod_{\alpha\in A}X_\alpha$.
By $\sigma(X,a)$ we denote the space
$$\left\{ x\in \prod_{\alpha\in A} X_\alpha : \{\alpha\in A: x_\alpha \neq a_\alpha\} \text{ is finite} \right\}.$$ 
A space of the form $\sigma(X,a)$ is called a \textit{$\sigma$-product}\footnote{In \cite{N} $\sigma$-products are called $\Sigma^0$-products
(cf. Definition \ref{sigma-cont-def} and Theorem \ref{Noble}).
It seems, however, that the term ``$\sigma$-product'' is more common in the literature.}.

If $x\in \sigma(X,a)$, we define the support of $x$ as
$$\supp(x)=\{\alpha\in A:x_\alpha\neq a_\alpha\}.$$
Note that by definition of $\sigma$-product, the set $\supp(x)$ is finite for any $x\in \sigma(X,a)$.

\begin{defin}\label{sigma-cont-def} (cf. \cite[p. 188]{N})
A function $f:\prod_{\alpha\in A} X_\alpha\to Z$ is called:
\begin{itemize}
 \item \textit{$2$-continuous} if it is continuous when restricted to any subset of the form $\prod_{\alpha\in A}Y_\alpha$, where $1\leq|Y_\alpha|\leq 2$
 for any $\alpha\in A$;
 \item \textit{$\sigma$-continuous} if its restriction to any $\sigma$-product in $\prod_{\alpha\in A} X_\alpha$ is continuous.
\end{itemize}
\end{defin}

Concerning the continuity of functions on Cartesian products, the following useful result was proved by Noble: 

\begin{thrm}\cite[Theorem 1.1]{N}\label{Noble}
Let $Z$ be a regular space. If $f:\prod_{\alpha\in A} X_\alpha\to Z$ is $\sigma$-continuous and $2$-continuous, then $f$ is continuous.  
\end{thrm}

\subsection{Minitightness of finite products}

It was shown by Okunev in \cite{O} that minitightness behaves nicely under finite products of compact spaces.
We have:

\begin{thrm}\cite[Theorem 2.14]{O}
If $X$ is locally compact, then for every Tychonoff space $Y$ we have $t_m(X\times Y)\leq t_m(X)t_m(Y)$. 
\end{thrm}
Since $t_m(X)=t_0(X)$ for a normal (in particular compact) space $X$,
from the above theorem, we immediately obtain:
\begin{cor}\label{Okunev}(Okunev)
Let $n\in \NN$ and let $\kappa$ be an infinite cardinal.
If $X_k$ is compact and $t_0(X_k)\leq\kappa$, for $k=0,\ldots , n$, then $t_0(\prod_{k\leq n} X_k)\leq\kappa$. 
\end{cor}

\section{Proofs}

Our proof of Theorem \ref{main} will be based on two lemmas given below.

\begin{lemma}\label{2-continuity}
Let $\kappa$ be an infinite cardinal and let $\{X_\alpha:\alpha<2^\kappa\}$ be a family of topological spaces.
If $f:\prod_{\alpha<2^\kappa} X_\alpha \to \RR$ is $\kappa$-continuous, then it is $2$-continuous.
\end{lemma}
\begin{proof}
Let $Y=\prod_{\alpha<2^\kappa} Y_\alpha$, where $Y_\alpha\subseteq X_\alpha$ is nonempty and has at most two elements. By the Hewitt-Marczewski-Pondiczery theorem
\cite[2.3.15]{E} $d(Y)\leq\kappa$. The function $f\upharpoonright Y:Y\to \RR$ is $\kappa$-continuous and since
$t_0(Y)\leq d(Y)$ \cite[Corollary 1]{A}, the result follows.
\end{proof}

\begin{lemma}\label{sigma-continuity}
Let $\kappa,\lambda$ be infinite cardinals and let $\{X_\alpha:\alpha<\lambda\}$ be a family of compact spaces such that $t_0(X_\alpha)\leq\kappa$, for every
$\alpha<\lambda$. 
If $f:\prod_{\alpha<\lambda} X_\alpha\to \RR$ is $\kappa$-continuous, then it is $\sigma$-continuous.  
\end{lemma}
\begin{proof}
Let $X=\prod_{\alpha<\lambda}X_\alpha$.
Striving for a contradiction, assume that there are: a $\sigma$-product $Y=\sigma(X,a)$, for some $a\in X$, and an open set $U\subseteq \RR$
such that $f^{-1}(U)\cap Y$ is not open in $Y$ (i.e. $f$ restricted to $Y$ is not continuous). This means that there is $y\in f^{-1}(U)\cap Y$ such that
\begin{align}\label{1}
f(W\cap Y)\nsubseteq U,\quad \text{for any open neighborhood $W$ of $y$ in $\prod_{\alpha<\lambda}X_\alpha$}.
\end{align}
Let $V\subseteq \overline{V}\subseteq U$ be an open neighborhood of $f(y)$.

\medskip

We shall inductively construct:
\begin{itemize}
 \item a sequence $(y^n)_{n\in \NN}$ of elements of the $\sigma$-product $Y$,
 \item a sequence $(A^n)_{n\in\NN}$ of finite subsets of $\lambda$ and
 \item a family $\{U^n_\alpha\subseteq X_\alpha:n\in \NN,\;\alpha\in A^n\}$ of open sets,
\end{itemize}
satisfying the following conditions:

\medskip

\begin{enumerate}[(i)]\itemsep9pt
 \item $y^0=y$ and $f(y^n)\notin U$, for $n\geq 1$;
 \item $A^0=\supp(y)$ and $A^n=A^{n-1}\cup\supp(y^n)$, for $n\geq 1$;%$A^0=\supp(y)\cup\{0\}$ and $A^n=A^{n-1}\cup \supp(y^n)\cup\{n\}$, for $n\geq 1$;
 \item $U_\alpha^n$ is an open subset of $X_\alpha$, for $n\geq 0$ and $\alpha\in A^n$;
 \item $U^n_\alpha\supseteq U^{n+1}_\alpha$, for $n\geq 0$ and $\alpha\in A^n$;
 \item $f\left( \prod_{\alpha\in A^n} \overline{U^n_\alpha}\times \prod_{\alpha\notin A^n}\{a_\alpha\} \right)\subseteq V$, for $n\geq 1$;
 \item $y\in\prod_{\alpha\in A^{n}} U^n_\alpha\times \prod_{\alpha\notin A^{n}}\{a_\alpha\}$, for $n\geq 0;$
 \item $y^n\in \left( \prod_{\alpha\in A^{n-1}} U^{n-1}_\alpha\times \prod_{\alpha\notin A^{n-1}} X_\alpha\right)\cap Y$, for $n\geq 1$.
\end{enumerate}

\medskip

For $n=0$ we put $y^0=y$, $A^0=\supp(y)$ and $U^0_\alpha=X_\alpha$, for $\alpha\in A^0$.
It is easy to check that conditions (i)--(vii) are satisfied for $n=0$.

Now, fix $m\in \NN$ and suppose that, for every $n\leq m$, we have constructed:
$y^n\in Y$, a finite set $A^n\subseteq \lambda$ and a family $\{U^n_\alpha:\alpha\in A^n\}$ satisfying (i)--(vii) for all $n\leq m$.
We shall construct $y^{m+1}$ in such a way that conditions (i)--(vii) remain true for $n\leq m+1$. To this end,
consider the set
$$W_m=\prod_{\alpha\in A^m} U^m_\alpha \times \prod_{\alpha\notin A^m} X_\alpha.$$
By the inductive assumption, condition (vi) holds for $n=m$ and thus we infer that the
set $W_m$ is an open neighborhood of $y$ in $\prod_{\alpha<\lambda} X_\alpha$. Hence, by \eqref{1}, there is
\begin{align}\label{2}
y^{m+1}\in W_m\cap Y \quad\text{such that}\quad f(y^{m+1})\notin U.
\end{align}

Let
\begin{align}\label{3}
A^{m+1}=A^m\cup\supp(y^{m+1}).
\end{align}
To define open sets $U^{m+1}_\alpha$, for $\alpha\in A^{m+1}$, we proceed as follows:

Let
$$K_{A^{m+1}}=\prod_{\alpha\in A^{m+1}}X_\alpha\times \prod_{\alpha\notin A^{m+1}}\{a_\alpha\}.$$
The set $K_{A^{m+1}}$ can be identified with the finite product $\prod_{\alpha\in A^{m+1}}X_\alpha$ and thus, by Corollary \ref{Okunev}, we infer that
$f\upharpoonright K_{A^{m+1}}$ is continuous. Moreover, by \eqref{3} (and the inductive assumption), $y\in K_{A^{m+1}}$. Therefore,
there is a basic open neighborhood of $y$ in $K_{A^{m+1}}$ whose closure is mapped by $f$ into $V$, i.e.
for
some open sets $U^{m+1}_\alpha\subseteq X_\alpha$, where $\alpha\in A^{m+1}$, we have
\begin{align}\label{4}
y\in \prod_{\alpha\in A^{m+1}}U^{m+1}_\alpha\times \prod_{\alpha\notin A^{m+1}}\{a_\alpha\} \quad \text{and} \quad
f\left( \prod_{\alpha\in A^{m+1}}\overline{U^{m+1}_\alpha}\times \prod_{\alpha\notin A^{m+1}}\{a_\alpha\}  \right)\subseteq V.
\end{align}
Again by \eqref{3}, $A^m\subseteq A^{m+1}$ and since (vi) holds for $n=m$, we may additionally assume that
\begin{align}\label{5}
U^m_\alpha\supseteq U^{m+1}_\alpha\quad \text{for}\quad \alpha\in A^m.
\end{align}
Now, for any $n\leq m+1$, conditions (i)--(vii) follow from (2)--(5) and the inductive assumption.
This finishes the inductive construction.

\medskip

Note that, for any $n\in \NN$, we have $\supp(y^n)\subseteq A^n$ (by (ii)) and thus, it follows from (i) and (v) that $y^n_\alpha\notin U^n_\alpha$,
for some $\alpha\in A^n$.
On the other hand, by (vii), $y^{n+1}_\alpha\in U^n_\alpha$, for any $n\in \NN$ and $\alpha\in A^n$. This implies that $y^{n+1}\notin \{y^0,\ldots ,y^n\}$ and hence
the set $\{y^n:n\in \NN\}$ is infinite.

\medskip

The space $X=\prod_{\alpha< \lambda}X_\alpha$ is compact and thus the (infinite) set $\{y^n:n\in\NN\}$ has a complete accumulation point $x\in X$.
For $n\in \NN$, let
\begin{equation*}
x^n_\alpha=
  \left\{\begin{aligned}
  &   x_\alpha\quad &\text{if}\quad \alpha\in A^n\\
&a_\alpha  &\text{if}\quad \alpha\notin A^n
  \end{aligned}
 \right. 
\end{equation*}
and let $x^n=(x^n_\alpha)_{\alpha<\lambda}\in \prod_{\alpha<\lambda}X_\alpha$.

Observe that
\begin{align}\label{6}
\text{for any } \alpha<\lambda \text{ we have } x_\alpha=\lim_{n\to\infty} x^n_\alpha,
\end{align}
i.e. $x$ is the pointwise limit of the sequence $(x^n)_{n\in \NN}$. Indeed, if
$\alpha\in \bigcup A^n$ then
$\lim_{n\to \infty} x^n_\alpha=x_\alpha$ since $A^n\subseteq A^{n+1}$, by (ii).
If $\alpha\notin \bigcup A^n$ then by definition of $x^n_\alpha$ we have
$x^n_\alpha=a_\alpha$ for every $n\in \NN$. Moreover, by (ii), $\alpha\notin \bigcup \supp(y^n)$ which means that
$y^n_\alpha=a_\alpha$, for every $n\in \NN$. Since $x$ is a complete accumulation point of $\{y^n:n\in\NN\}$ we infer that
$x_\alpha=a_\alpha$ too.

\medskip

\textbf{Claim.} For every $n\geq 0$ we have
$$x^n\in\left( \prod_{\alpha\in A^n}\overline{U^n_\alpha}\times \prod_{\alpha\notin A^n}\{a_\alpha\} \right).$$
\begin{proof}
By definition of $x^n$ we clearly have $x^n_\alpha=a_\alpha$ for $\alpha\notin A^n$. We need to show that $x^n_\alpha=x_\alpha\in \overline{U^n_\alpha}$,
for any $\alpha\in A^n$.

Striving for a contradiction, suppose that $x^n_\alpha=x_\alpha\notin \overline{U^n_\alpha}$, for some $\alpha\in A^n$.
Then, the set
$$M=\left(\prod_{\beta\neq \alpha}X_\beta \times \left( X_\alpha\setminus \overline{U^n_\alpha} \right) \right)$$
is an open neighborhood of $x$ in $X$. Since $x$ is a complete accumulation point of $\{y^n:n\in \NN\}$, there is an infinite set $I\subseteq\NN$ such that
$\{y^i:i\in I\}\subseteq M$. This, however, means that
$$y^i_\alpha\in X_\alpha\setminus \overline{U^n_\alpha},\quad \text{for any}\quad i\in I.$$
Take $j\in I$ with $j>n$ (recall that $I$ is infinite). We have
$$y^j_\alpha\in X_\alpha\setminus \overline{U^n_\alpha}.$$ But $j>n$, so $A_{j-1}\supseteq A_n\ni \alpha$ and by (vii) and (iv),
$$y^j_\alpha\in U^{j-1}_\alpha\subseteq U^n_\alpha,$$
a contradiction.
\end{proof}

By assumption, the function $f:X\to \RR$ is $\kappa$-continuous and thus it is continuous when restricted to the countable set $\{x^n:n\in \NN\}\cup \{x\}$.
Since, by Claim and condition (v), we have $f(x^n)\in V$ for any $n\in \NN$, we infer (using \eqref{6}) that $f(x)\in \overline{V}\subseteq U$.

On the other hand, $f$ restricted to the countable set $\{y^n:n\in \NN\}\cup\{x\}$ is continuous (again by the $\kappa$-continuity of $f$).
By (i), $f(y^n)\notin U$ for $n\geq 1$ contradicting the fact that $x$ is a complete accumulation point of $\{y^n:n\in \NN\}$.
\end{proof}

Now, Theorem \ref{main} follows directly from Lemmas \ref{2-continuity} and \ref{sigma-continuity} and Theorem \ref{Noble}.
In particular, we obtain the following result which affirmatively answers a question posed by Okunev \cite[Question 3.2]{O}.
\begin{cor}\label{AnswerOkunev}
$t_0(X)=t_m(X)=t_m(X^\mathfrak{c})=t_0(X^\mathfrak{c})$ for any compact space $X$. 
\end{cor}
\begin{proof}
The minitightness and the tightness are equal for any compact space so $t_0(X)=t_m(X)$ and $t_m(X^\mathfrak{c})=t_0(X^\mathfrak{c})$.
Since $t_0(X)\geq \omega$,
by Theorem \ref{main} we have $t_0(X)\geq t_0(X^\mathfrak{c})$. The opposite inequality follows from the fact that the functional tightness
cannot be raised by quotient mappings \cite[Proposition 3]{A}.
\end{proof}

Let us recall the following theorem proved by Uspenskii \cite{U} concerning the functional tightness of Cantor cubes
and involving measurable cardinals.
\begin{thrm}\cite[Theorem 2]{U}\label{cubes}
If $\kappa$ is less than the first measurable cardinal, then any $\omega$-continuous function $f:\{0,1\}^\kappa\to \mathbb{R}$, is continuous.
In other words, $t_0(\{0,1\}^\kappa)=\omega$ provided $\kappa$ is below the first measurable cardinal.
\end{thrm}

Observe that the above result is a refinement of Lemma \ref{2-continuity} for $\kappa=\omega$. Applying Lemma \ref{sigma-continuity}
and Theorems \ref{cubes} and \ref{Noble} we get:

\begin{thrm}\label{main2}
Let $\{X_\alpha:\alpha<\lambda\}$ be a family of compact spaces
such that $t_0(X_\alpha)\leq \kappa$. If $\lambda$ is less than the first measurable cardinal,
then $t_0\left( \prod_{\alpha<\lambda} X_\alpha \right)\leq \kappa$.
\end{thrm}

In particular, if there are no measurable cardinals the functional tightness is preserved by arbitrarily large products of compacta.
We can refine Corollary \ref{AnswerOkunev} in the following way:

\begin{cor}
If $\kappa$ is less than the first measurable cardinal then 
$t_0(X)=t_m(X)=t_m(X^\kappa)=t_0(X^\kappa)$ for any compact space $X$. 
\end{cor}

Combining Theorems \ref{main} and \ref{main2} we obtain:
\begin{cor}\label{main1}
Let $\{X_\alpha:\alpha<\lambda\}$ be a family of compact spaces
such that $t_0(X_\alpha)\leq \kappa$. If $\lambda \leq 2^\kappa$ or $\lambda$ is less than the first measurable cardinal,
then $t_0\left( \prod_{\alpha<\lambda} X_\alpha \right)\leq \kappa$.
\end{cor}

\begin{xrem}
The non-measurability of $\lambda$ in the formulation of Theorem \ref{main2} cannot be eliminated.
Indeed, assume that $\lambda$ is measurable and let $\mu$ be a two-valued nontrivial measure on $\lambda$. Identifying subsets of $\lambda$ with
elements of the Cantor cube $2^\lambda$ (via the map $A\mapsto \chi_A$)
we can treat $\mu$ as a function $\mu:2^\lambda\to \{0,1\}$. It is not difficult to check that this function is
$\omega$-continuous but not continuous, cf. \cite[Theorem 4.1]{Pl1} and remark following Corollary 3 in \cite{U}.
\end{xrem}

\subsection*{Acknowledgements}
I should like to thank Witold Marciszewski and Grzegorz Plebanek for several valuable comments.

\end{document}